\documentclass[10pt]{article}
\usepackage[T1]{fontenc}
\usepackage[cp1250]{inputenc}
\usepackage{lmodern}

\usepackage[english]{babel}
\usepackage{algorithm}
\usepackage{algpseudocode}
\usepackage{latexsym}
\usepackage{amsmath}
\usepackage{indentfirst}
\usepackage{amsthm}
\usepackage{amssymb}
\usepackage{amsfonts}
\usepackage{stackrel}
\usepackage{dsfont}
\usepackage{tikz}
\usepackage[mathscr]{eucal}
\usepackage{authblk}
\usepackage{hyperref}
\usepackage{mathabx}
\usepackage{bookmark}
\usepackage{listings} 

\newtheorem{thm}{Theorem}

\newtheorem{cor}[thm]{Corollary}

\newcommand{\reals}{\mathbb{R}}

\newcommand{\Prob}{\mathbb{P}}

\begin{document}
\title{ROC curves for LDA classifiers}
\author{Mateusz Krukowski}
\affil{Institute of Mathematics, \L\'od\'z University of Technology, \\ W\'ol\-cza\'n\-ska 215, \
90-924 \ \L\'od\'z, \ Poland \\ \vspace{0.3cm} e-mail: mateusz.krukowski@p.lodz.pl}
\maketitle

\begin{abstract}
In the paper, we derive an analytic formula for the ROC curves of the LDA classifiers. We establish elementary properties of these curves (monotonicity and concavity), provide formula for the area under curve (AUC) and compute the Youden J-index. Finally, we illustrate the performance of our results on a real--life dataset of Wisconsin breast cancer patients. 
\end{abstract}

\vspace{0.2cm}

\section{Introduction}
\label{section:Introduction}

Classification is one of the fundamental problems in data analysis (see Chapters 5--8 in \cite{Bonaccorso} or \cite{Flach}, \cite{Herbrich}, \cite{MichieSpiegelhalterTaylor}). In its binary version, which we focus on in the current paper, the task is to distinguish elements coming from two classes (referred to as negative and positive), given a certain amount of their features. For instance, we wish to 
\begin{itemize}
	\item predict whether a student will pass or fail an exam based on how long they have studied, or   
	\item tell apart malignant and benign tumors by looking at their mean area, concavity and fractal dimension, or
	\item diagnose diabetes by examining patient's average blood pressure, cholesterol, blood sugar level etc.
\end{itemize}

\noindent
Throughout the paper, we assume that there are $n$ features, each one being a continuous variable. Hence, the whole feature space is identified with $\reals^n.$

The negative class is customarily indexed with $0$, while the positive class is indexed with $1$. The classes need not be balanced, so that an element is drawn from the negative class with probability $p_0$ and from the positive class with probability $p_1$ (not necessarily equal to $p_0$). Obviously, $p_0 + p_1 = 1,$ since the classification is binary, i.e., there are only two classes.

Given a function $f:\reals^n\longrightarrow \reals$ and a threshold space $\Theta\subset \reals,$ we define a family of binary classifiers $C_{\theta}:\reals^n\longrightarrow \{0,1\},\ \theta\in\Theta$ via the formula 
$$C_{\theta}(x) := \mathds{1}_{\{f > \theta\}}(x).$$

\noindent
For instance, the function 
$$f_{logistic}(x; \alpha,\beta) := \frac{1}{1+e^{-(\alpha + \langle \beta|x\rangle)}}$$

\noindent
and threshold space $\Theta_{logisitc} = [0,1]$ define a family of logistic regression classifiers. Parameters $\alpha,\beta$ are chosen by maximizing (usually via some sort of gradient descent) the likelihood function (see Chapter 4 in \cite{Geron} or Chapter 16 in \cite{Grus}). Another example is 
$$f_{kNN}(x) := \frac{\sum_{y \in S(x)}\ y}{k}$$

\noindent
and $\Theta_{kNN} := \{0,1,\ldots,k\},$ where $S(x)$ is a set of labels (either $0$ or $1$) of those $k$ points in the data set, which are closest to $x$. The pair $(f_{kNN}, \Theta_{kNN})$ gives rise to a family of k--nearest neighbors classifiers (see Chapter 12 in \cite{Grus}). 

The current paper is devoted to the family of LDA classifiers and we will review their fundamentals at the beginning of Section \ref{section:ROC}. For the time being, let us suppose that we have decided to use a particular family of classifiers (with a specific threshold space $\Theta$), be it logistic regression, kNN, LDA, SVM etc. We wish to measure the model's performance and thus, we distinguish four possible scenarios: 
\begin{itemize}
	\item Model correctly predicts the negative class, i.e., it predicts label $0$ when the test data point's true label is $0$. We say that the prediction is true negative. 
	\item Model incorrectly predicts the positive class, i.e., it predicts label $1$ when the test data point's true label is $0$. We say that the prediction is false positive.
	\item Model incorrectly predicts the negative class, i.e., it predicts label $0$ when the test data point's true label is $1$. We say that the prediction is false negative.
	\item Model correctly predicts the positive class, i.e., it predicts label $1$ when the test data point's true label is $1$. We say that the prediction is true positive.
\end{itemize}

\noindent
Let $M_{\theta}:\reals^n\longrightarrow \{TN, FP, FN, TP\}$ be a family of random variables which depend on the threshold $\theta\in\Theta$ and describe four possible scenarios above. In practice, we never know the true distribution of $M_{\theta}$ (i.e., the values $\Prob(M_{\theta} =  TN),\ \Prob(M_{\theta} = FP)$ etc.) and thus have to resort to approximations on the basis of the provided sample. One of the main points of our paper is that the family of LDA classifiers may be an exception in this regard. This is due to the fact that at its core, LDA \textit{assumes} a particular form of distribution of the classes and thus, indirectly, the distribution of $M_{\theta}.$ Theorem \ref{maintheorem} in Section \ref{section:ROC} is devoted to the computation of this distribution. For now, we make an obvious remark that the fewer false predictions a classifier makes the better it is. This motives the introduction of false positive rate ($FPR$) and true positive rate ($TPR$) defined by  
\begin{equation}
\begin{split}
FPR(\theta) &:= \frac{\Prob(M_{\theta} = FP)}{\Prob(M_{\theta} = FP) + \Prob(M_{\theta} = TN)},\\
TPR(\theta) &:= \frac{\Prob(M_{\theta} = TP)}{\Prob(M_{\theta} = TP) + \Prob(M_{\theta} = FN)}.
\end{split}
\label{fprtpr}
\end{equation}

\noindent
Let us observe that a threshold $\theta\in\Theta$ such that $FPR(\theta) = 0$ and $TPR(\theta) = 1$ corresponds to a classifier, which is never mistaken, i.e., $\Prob(M_{\theta} = FP) = \Prob(M_{\theta} = FN) = 0.$ Obviously, such a scenario does not occur in practice, but it still serves as an ideal reference point. 

In summary, given a whole slew of binary classifiers $C_{\theta}, \theta\in\Theta$ we are interested in finding the one, for which the FPR and TPR are (in some sense, which we will discuss later) close to $0$ and $1$, respectively. In order to achieve this goal, we study the whole set 
$$\{(FPR(\theta), TPR(\theta)) \ :\ \theta \in \Theta\},$$ 

\noindent
called the ROC curve. 

\section{LDA classifier and its theoretical ROC curve}
\label{section:ROC}

Amongst various classifiers such as the logistic regression, k--nearest neighbours (kNN) and support vector machine (SVM), linear discriminant analysis (LDA) stands out due to its underlying assumption. The foundational premise of the LDA classifier is that both classes exhibit multivariate normal distribution with the same covariance matrix $\Sigma.$ In other words, the probability density function of class $i$ (with $i$ equal to either $0$ or $1$) is given by
$$\phi_i(x) := (2\pi)^{-\frac{n}{2}} \det(\Sigma)^{-\frac{1}{2}} \exp\left( - \frac{1}{2} (x-\mu_i)^t \Sigma^{-1} (x-\mu_i)\right),$$

\noindent
where
\begin{itemize}
	\item $n$ is the number of dimensions of the data set (i.e., number of features), 
	\item $\mu_i$ is the vector mean of class $i.$
\end{itemize}

\noindent
The family of LDA classifiers is defined by the function 
$$f_{LDA}(x) := \log \frac{\phi_1(x)}{\phi_0(x)} = (\mu_1-\mu_0)^t\Sigma^{-1}x + 0.5\bigg(\mu_0^t\Sigma^{-1}\mu_0 - \mu_1^t\Sigma^{-1}\mu_1\bigg)$$ 

\noindent
and the threshold space $\Theta_{LDA} := \reals.$ For convenience, we put 
\begin{equation}
\begin{split}
\alpha &:= \Sigma^{-1}(\mu_1 - \mu_0),\\
\beta &:= 0.5\bigg(\mu_0^t\Sigma^{-1}\mu_0 - \mu_1^t\Sigma^{-1}\mu_1\bigg),
\end{split}
\label{alphaandbeta}
\end{equation}

\noindent
so that $f_{LDA}(x) = \alpha^t x + \beta.$ Furthermore, we use $H_1(\theta)$ to denote the positive half-space $\alpha^t x + \beta > \theta$ and $H_0(\theta)$ to denote the negative half-space $\alpha^t x + \beta < \theta$. 

At this stage, having reviewed the basics of LDA classifiers, we come back to the distribution of $M_{\theta},$ partially discussed in the Introduction. We observe that 
\begin{equation*}
\begin{split}
\Prob(M_{\theta} = TN) &= \Prob(\text{point lies in } H_0(\theta) \text{ and belongs to class } 0) \\
&= \Prob(\text{point belongs to class } 0) \Prob(\text{point lies in } H_0(\theta) | \text{point belongs to class } 0) \\
&= p_0 \int_{H_0(\theta)}\ \phi_0(x)\ dx.
\end{split}
\end{equation*}

\noindent
Similarly, we have  
\begin{equation*}
\begin{split}
\Prob(M_{\theta} = FP) &= p_0 \int_{H_1(\theta)}\ \phi_0(x)\ dx,\\
\Prob(M_{\theta} = FN) &= p_1 \int_{H_0(\theta)}\ \phi_1(x)\ dx,\\
\Prob(M_{\theta} = TP) &= p_1 \int_{H_1(\theta)}\ \phi_1(x)\ dx,
\end{split}
\end{equation*}

\noindent
so the next step is to devise a theorem enabling the calculation of the above integrals. 

\begin{thm}\label{maintheorem}
Let $X$ be a random variable with multivariate normal distribution, i.e., $X \sim N(\mu, \Sigma)$ where $\mu\in\reals^n$ is the mean and $\Sigma$ is a positive-definite (hence symmetric) covariance matrix. If $H$ is a half-space given by $\alpha^t x + \beta < 0$ where $\alpha\in \reals^n, \beta \in \reals,$ then
$$\int_H\ \phi_X(x)\ dx = \Phi\left(-\frac{\alpha^t\mu + \beta}{\|\sqrt{\Lambda}Q^t\alpha\|}\right),$$

\noindent
where 
\begin{itemize}
	\item $\phi_X$ is the probability density function of $X$,
	\item $\Phi$ is the cumulative distribution function of the univariate, standardized normal distribution,
	\item $\Sigma = Q\Lambda Q^t$ is a spectral decomposition, i.e., $Q$ is an orthogonal matrix, whose columns are  eigenvectors of $\Sigma$ and $\Lambda$ is a diagonal matrix with the corresponding eigenvalues. 
\end{itemize}
\end{thm}
\begin{proof}
We have 
\begin{equation}
\begin{split}
\int_H\ \phi_X(v)\ dv &= (2\pi)^{-\frac{n}{2}} \det(\Sigma)^{-\frac{1}{2}} \int_H\ \exp\left( - \frac{1}{2} (v-\mu)^t \Sigma^{-1} (v-\mu)\right)\ dv \\
&\stackrel{w := v-\mu}{=} (2\pi)^{-\frac{n}{2}} \det(\Sigma)^{-\frac{1}{2}} \int_{\{w\in\reals^n\ :\ \alpha^t(w + \mu) + \beta < 0\}}\ \exp\left( - \frac{1}{2} w^t \Sigma^{-1}w\right)\ dw\\
&= (2\pi)^{-\frac{n}{2}} \det(\Sigma)^{-\frac{1}{2}} \int_{\{w\in\reals^n\ :\ \alpha^t(w + \mu) + \beta < 0\}}\ \exp\left( - \frac{1}{2} (Q^tw)^t \Lambda^{-1} (Q^tw)\right)\ dw \\
&\stackrel{x := Q^tw}{=} (2\pi)^{-\frac{n}{2}} \det(\Sigma)^{-\frac{1}{2}} \int_{\{x\in\reals^n\ :\ \alpha^t(Qx + \mu) + \beta < 0\}}\ \exp\left( - \frac{1}{2} x^t \Lambda^{-1} x\right)\ dx \\
&\stackrel{y := \sqrt{\Lambda}^{-1}x}{=} (2\pi)^{-\frac{n}{2}} \int_{\{y\in\reals^n\ :\ \alpha^t(Q\sqrt{\Lambda}y + \mu) + \beta < 0\}}\ \exp\left( - \frac{1}{2} \|y\|^2\right)\ dy \\
&= (2\pi)^{-\frac{n}{2}} \int_{\{y\in\reals^n\ :\ \gamma^t y + \alpha^t\mu + \beta < 0\}}\ \exp\left( - \frac{1}{2} \|y\|^2\right)\ dy
\end{split}
\label{long1}
\end{equation}

\noindent
where $\gamma := \sqrt{\Lambda} Q^t\alpha$. Next, let $P$ be an orthogonal matrix such that $P\gamma = \|\gamma\| e_1.$ Using the substitution $z := Py$ in \eqref{long1} we obtain
\begin{equation*}
\begin{split}
\int_H\ \phi_X(v)\ dv &= (2\pi)^{-\frac{n}{2}} \int_{\{z\in\reals^n\ :\ \gamma^t P^t z + \alpha^t\mu + \beta < 0\}}\ \exp\left( - \frac{1}{2} \|z\|^2\right)\ dz\\
&= (2\pi)^{-\frac{n}{2}} \int_{\{z\in\reals^n\ :\ \|\gamma\| z_1 + \alpha^t\mu + \beta < 0\}}\ \exp\left( - \frac{1}{2} \left(z_1^2 + \sum_{k=2}^n\ z_k^2\right)\right)\ dz\\
&= (2\pi)^{-\frac{1}{2}} \int_{-\infty}^{-\frac{(\alpha^t\mu + \beta)}{\|\gamma\|}}\ \exp\left(-\frac{1}{2}z_1^2\right)\ dz_1 \cdot (2\pi)^{-\frac{n-1}{2}} \int_{\reals^{n-1}}\ \exp\left(-\frac{1}{2}\sum_{k=2}^n\ z_k^2\right)\ dz_2\ldots dz_n\\
&= \Phi\left(-\frac{\alpha^t\mu + \beta}{\|\sqrt{\Lambda}Q^t\alpha\|}\right),
\end{split}
\end{equation*}

\noindent
which concludes the proof.
\end{proof}

The following result is a simple application of Theorem \ref{maintheorem} and determines the distribution of $M_{\theta}$:
\begin{cor}
We have 
\begin{equation*}
\begin{split}
\Prob(M_{\theta} = TN) &= p_0 \Phi\left(-\frac{\alpha^t\mu_0 + \beta - \theta}{\|\sqrt{\Lambda}Q^t\alpha\|}\right),\\
\Prob(M_{\theta} = FP) &= p_0 \left(1 - \Phi\left(-\frac{\alpha^t\mu_0 + \beta - \theta}{\|\sqrt{\Lambda}Q^t\alpha\|}\right)\right),\\
\Prob(M_{\theta} = FN) &= p_1 \Phi\left(-\frac{\alpha^t\mu_1 + \beta - \theta}{\|\sqrt{\Lambda}Q^t\alpha\|}\right),\\
\Prob(M_{\theta} = TP) &= p_1 \left(1-\Phi\left(-\frac{\alpha^t\mu_1 + \beta - \theta}{\|\sqrt{\Lambda}Q^t\alpha\|}\right)\right).
\end{split}
\end{equation*}
\end{cor}

\noindent
Furthermore, Theorem \ref{maintheorem} allows for the computation of FPR and TPR, and consequently, the ROC curve (we omit an elementary derivation):
\begin{cor}
\noindent
We have
\begin{equation*}
\begin{split}
FPR(\theta) = 1 - \Phi\left(-\frac{\alpha^t\mu_0 + \beta - \theta}{\|\sqrt{\Lambda}Q^t\alpha\|}\right),\\
TPR(\theta) = 1-\Phi\left(-\frac{\alpha^t\mu_1 + \beta - \theta}{\|\sqrt{\Lambda}Q^t\alpha\|}\right).
\end{split}
\end{equation*}

\noindent
Consequently, $FPR(\theta), TPR(\theta) \in (0,1)$ for every $\theta\in\reals$ and the following equality holds true:
\begin{gather}
TPR = 1- \Phi\left(\Phi^{-1}(1-FPR) - \frac{\alpha^t(\mu_1 -\mu_0)}{\|\sqrt{\Lambda}Q^t\alpha\|}\right).
\label{ROCformula}
\end{gather}
\end{cor}

Trivially, equality \eqref{ROCformula} allows for the following formula for the area under the curve (AUC):
\begin{equation*}
\begin{split}
AUC &= 1 - \int_0^1\ \Phi\left(\Phi^{-1}(1-FPR) - \frac{\alpha^t(\mu_1 -\mu_0)}{\|\sqrt{\Lambda}Q^t\alpha\|}\right)\ d(FPR)\\
&\stackrel{\Phi^{-1}(1-FPR) = v}{=} 1 - \int_{-\infty}^{\infty}\ \Phi\left(v - \frac{\alpha^t(\mu_1 -\mu_0)}{\|\sqrt{\Lambda}Q^t\alpha\|}\right)\phi(v)\ dv.
\end{split}
\end{equation*}

Our next goal is to establish the fundamental properties of the ROC curve, namely monotonicity and concavity. This turns out to be a standard problem in mathematical analysis:

\begin{thm}
ROC curve \eqref{ROCformula} is increasing and concave.
\end{thm}
\begin{proof}
First, for notational convenience, we put $\delta := \frac{\alpha^t(\mu_1 -\mu_0)}{\|\sqrt{\Lambda}Q^t\alpha\|}.$ By the inverse function rule we have 
$$\frac{d(TPR)}{d(FPR)} = \frac{\phi(\Phi^{-1}(1-FPR) - \delta)}{\phi(\Phi^{-1}(1-FPR))} > 0,$$

\noindent
which means that the ROC curve is increasing. 

Next, we compute the second derivative:  
$$\frac{d^2(TPR)}{d(FPR)^2} = \frac{\phi(\Phi^{-1}(1-FPR) - \delta)}{\phi^2(\Phi^{-1}(1-FPR))} \left(\frac{\phi'(\Phi^{-1}(1-FPR))}{\phi(\Phi^{-1}(1-FPR))} - \frac{\phi'(\Phi^{-1}(1-FPR) - \delta)}{\phi(\Phi^{-1}(1-FPR) - \delta)}\right).$$

\noindent
Since $\frac{\phi'(t)}{\phi(t)} = -t$ for every $t\in\reals$, we conclude that 
$$\frac{d^2(TPR)}{d(FPR)^2} = -\frac{\phi(\Phi^{-1}(1-FPR) - \delta)}{\phi^2(\Phi^{-1}(1-FPR))}\delta,$$

\noindent
which establishes concavity of the ROC curve.
\end{proof}

Finally, we come back to the question of choosing the ``best'' threshold $\theta.$ There are various ways to understand ``optimality'' in this context, but Youden's J--index $J(\theta) := TPR(\theta) - FPR(\theta)$ is arguably one of the most popular (see \cite{Youden}). The corresponding Youden's threshold is given by
$$\theta_{Youden} := \arg \max_{\theta\in\Theta} (TPR(\theta) - FPR(\theta))$$

\noindent
and the following result establishes its value in the LDA setting:

\begin{thm}
We have $\theta_{Youden} = 0$ with the corresponding rates 
\begin{equation*}
\begin{split}
FPR(0) &= 1 - \Phi\left(-\frac{\alpha^t\mu_0 + \beta}{\|\sqrt{\Lambda}Q^t\alpha\|}\right),\\
TPR(0) &= 1-\Phi\left(-\frac{\alpha^t\mu_1 + \beta}{\|\sqrt{\Lambda}Q^t\alpha\|}\right)
\end{split}
\end{equation*}

\noindent
and the maximal value of the Youden J--index
$$J(0) = \Phi\left(-\frac{\alpha^t\mu_0 + \beta}{\|\sqrt{\Lambda}Q^t\alpha\|}\right) - \Phi\left(-\frac{\alpha^t\mu_1 + \beta}{\|\sqrt{\Lambda}Q^t\alpha\|}\right).$$
\end{thm}
\begin{proof}
For notational convenience, we set $\gamma := \sqrt{\Lambda}Q^t\alpha,\ \delta_0 := \frac{\alpha^t\mu_0 + \beta}{\|\gamma\|}$ and $\delta_1 := \frac{\alpha^t\mu_1 + \beta}{\|\gamma\|}.$ We have 
\begin{equation*}
\begin{split}
\frac{d(TPR-FPR)}{d\theta} &= \frac{\phi\left(\frac{\theta}{\|\gamma\|} - \delta_0\right) - \phi\left(\frac{\theta}{\|\gamma\|} - \delta_1\right)}{\|\gamma\|}\\
&= \frac{1}{2\pi\|\gamma\|} \left(\exp\left(-\frac{1}{2}\left(\frac{\theta}{\|\gamma\|} - \delta_0\right)^2\right) - \exp\left(-\frac{1}{2}\left(\frac{\theta}{\|\gamma\|} - \delta_1\right)^2\right)\right) \\
&\stackrel{\delta_1 = \delta_0 + \delta}{=} \frac{\exp\left(-\frac{1}{2}\left(\frac{\theta}{\|\gamma\|} - \delta_0\right)^2\right)}{2\pi\|\gamma\|} \left(1 - \exp\left(\delta\left(\frac{\theta}{\|\gamma\|}-\delta_0\right) - \frac{1}{2}\delta^2\right)\right),
\end{split}
\end{equation*} 

\noindent
which implies that the Youden's J-index is maximized if and only if  
$$1 = \exp\left(\delta\left(\frac{\theta}{\|\gamma\|}-\delta_0\right) - \frac{1}{2}\delta^2\right).$$

\noindent
We compute 
$$\theta_{Youden} = \alpha^t\left(\frac{\mu_0+\mu_1}{2}\right) + \beta \stackrel{\eqref{alphaandbeta}}{=} 0.5(\mu_1 - \mu_0)^t\Sigma^{-1}(\mu_0+\mu_1) + 0.5\bigg(\mu_0^t\Sigma^{-1}\mu_0 - \mu_1^t\Sigma^{-1}\mu_1\bigg),$$

\noindent
which, after a simple rearrangement of terms, concludes the proof. 
\end{proof}

\section{Application to Wisconsin breast cancer dataset}

This final section of the paper is focused on illustrating the performance of discussed results and techniques. To this end, we make use of the classic Wisconsin breast cancer dataset, available in the scikit--learn package in Python. The dataset was created in 1995 by Wolberg, Street and Mangasarian and consists of nearly 569 records of breast cancer patients. Each record compiles 30 characteristics of the tumor, i.e., its texture, area, smoothness, symmetry etc. The patients are divided into two classes, ``benign'' and ``malignant'', according to the type of cancer they have. 

For the sake of simplicity we restrict our attention to just two characteristics: mean area and smoothness. This allows for the visualization in Figure \ref{area_smoothness}, which gives partial credence to modeling the distributions of both classes as multivariate normal.  

\begin{figure}
\caption{Wisconsin breast cancer dataset -- mean area versus mean smoothness}
\centering
\includegraphics[width=8cm]{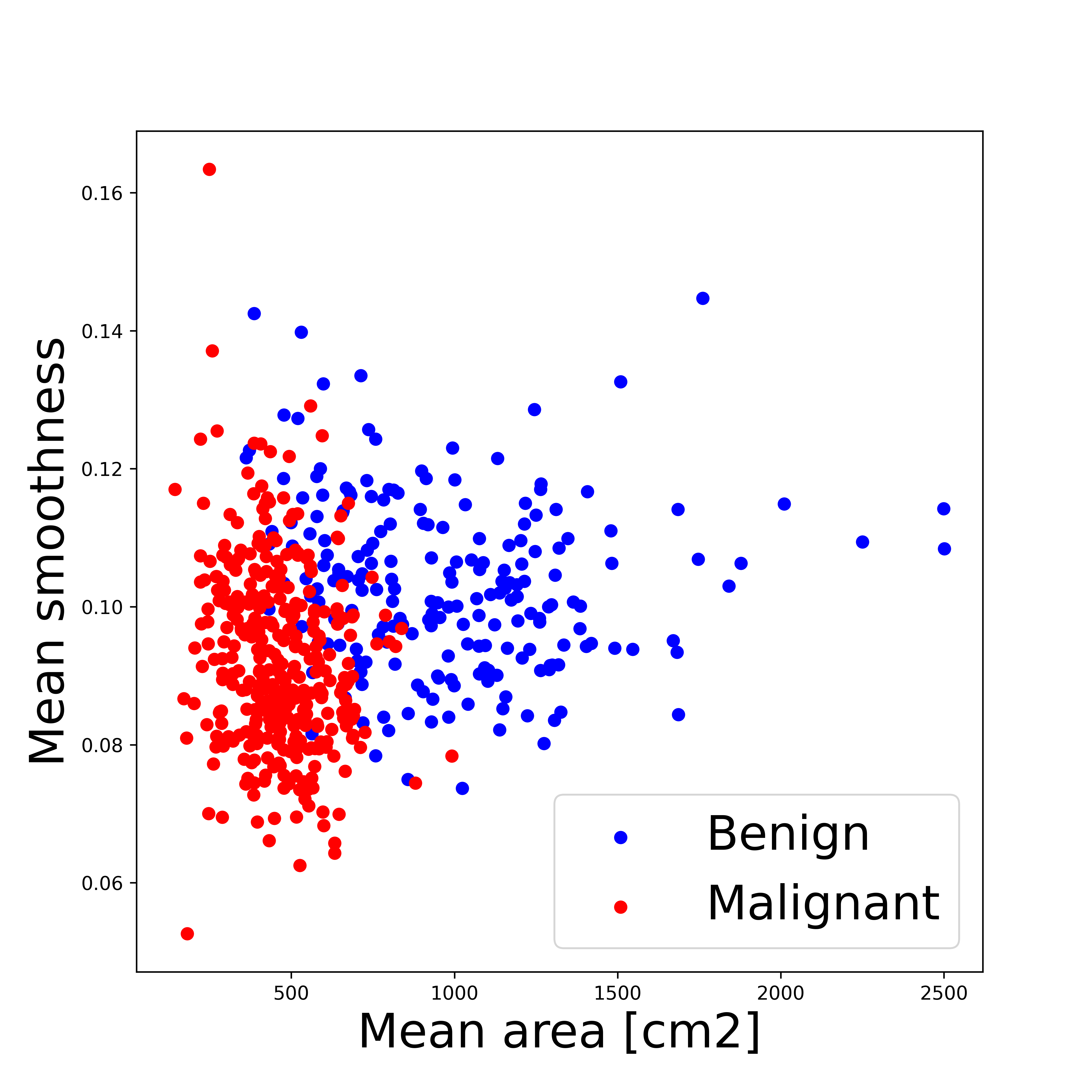}
\label{area_smoothness}
\end{figure}

In order to assess the performance of our results we begin by splitting the dataset into training and test sets in a ratio of $4:1$. Consequently, we estimate the means $\mu_0, \mu_1$, where $0$ corresponds to the negative, benign class, while $1$ denotes the positive, malignant class. Next, we determine the covariance matrix, its spectral decomposition and the values of $\alpha, \beta, \delta$. This allows for the computation of the theoretical ROC curve as well as its AUC (via numerical integration) and Youden coordinates $(FPR(0), TPR(0)).$ For comparison we also work out the empirical ROC and its corresponding AUC. All of these steps can be accomplished in roughly 60 lines of Python code, and the result is plotted in Figure \ref{ROC}:
\begin{figure}
\caption{Theoretical and empirical ROC curves with Youden J-indices}
\centering
\includegraphics[width=8cm]{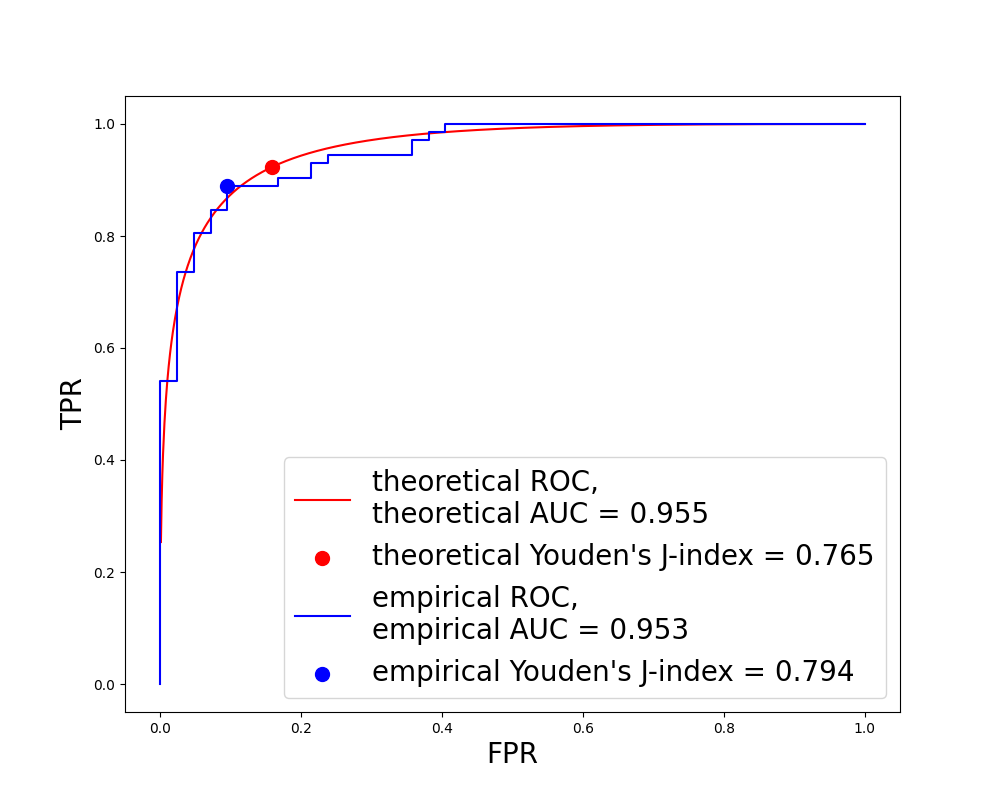}
\label{ROC}
\end{figure}

\begin{lstlisting}[language=Python]
import numpy as np
from sklearn.datasets import load_breast_cancer
from sklearn.model_selection import train_test_split
breast_cancer = load_breast_cancer(as_frame=True)
X = breast_cancer.data[['mean area', 'mean smoothness']].values
y = breast_cancer.target.values
X_train, X_test, y_train, y_test = \
         train_test_split(X, y, test_size=0.2, random_state=45)

from sklearn.discriminant_analysis import LinearDiscriminantAnalysis
lda = LinearDiscriminantAnalysis(solver = 'eigen',
                                 store_covariance = True)
lda.fit(X_train, y_train)
mu0, mu1 = lda.means_
Sigma = lda.covariance_
eigenvalues, Q = np.linalg.eigh(Sigma)
Lambda = np.diag(eigenvalues)
alpha = lda.coef_[0]
beta = lda.intercept_

from scipy.stats import norm
delta = np.dot(alpha, mu1-mu0) /\
        np.linalg.norm(np.sqrt(Lambda) @ Q.T @ alpha)
FPR = np.linspace(0.001, 0.999, 1000)
TPR = 1-norm.cdf(norm.ppf(1-FPR)-delta)

from scipy import integrate
AUC = integrate.trapezoid(TPR, dx=(0.999-0.001)/1000)

FPR0 = 1-norm.cdf(-(np.dot(alpha,mu0)+beta)/
                  np.linalg.norm(np.sqrt(Lambda) @ Q.T @ alpha))[0]
TPR0 = 1-norm.cdf(-(np.dot(alpha,mu1)+beta)/
                  np.linalg.norm(np.sqrt(Lambda) @ Q.T @ alpha))[0]

from sklearn.metrics import roc_curve, roc_auc_score
y_proba = lda.predict_proba(X_test)[:, 1]
FPR_emp, TPR_emp, _ = roc_curve(y_test, y_proba)
theta_emp = np.argmax(TPR_emp-FPR_emp)
J_emp = TPR_emp[theta_emp] - FPR_emp[theta_emp]
AUC_emp = roc_auc_score(y_test, y_proba)

import matplotlib.pyplot as plt
plt.figure(figsize=(10,8))
plt.plot(FPR, TPR, c='red', label='theoretical ROC,\n' + \
         f'theoretical AUC = {np.round(AUC,3)}')
plt.scatter(FPR0, TPR0, c='red', s=100,
            label=f'theoretical Youden\'s J-index = {np.round(TPR0-FPR0,3)}')
plt.plot(FPR_emp, TPR_emp, c='blue', label='empirical ROC,\n' + \
         f'empirical AUC = {np.round(AUC_emp,3)}')
plt.scatter(FPR_emp[theta_emp], TPR_emp[theta_emp], c='blue', s=100,
            label=f'empirical Youden\'s J-index = {np.round(J_emp,3)}')
plt.xlabel('FPR', fontsize=20)
plt.ylabel('TPR', fontsize=20)
plt.legend(fontsize=20)
plt.savefig('ROC.png')
\end{lstlisting}

\end{document}